\documentclass[11pt]{amsart}

\usepackage[utf8x]{inputenc}
\usepackage[english]{babel}

\usepackage{graphicx}
\usepackage{paralist}

\usepackage{amssymb,amsmath}
\usepackage[bookmarks,colorlinks,citecolor=magenta,pdftex]{hyperref}
\usepackage[dvipsnames]{xcolor}
\usepackage{tikz}
\usetikzlibrary{shapes}
\usepackage[a4paper,margin=2.7cm]{geometry}
\usepackage{bbm}

\newcommand\Defn[1]{\textbf{\color{blue}#1}}

\newcommand\vdotsHacked{\vphantom{\int^X}\smash{\vdots}}
\newcommand\ddotsHacked{\vphantom{\int^X}\smash{\ddots}}

\newcommand\R{\mathbb{R}}

\renewcommand\emptyset{\varnothing}

\newcommand\1{\mathbf{1}}

\newcommand\rc{\mathrm{rc}}
\newcommand\rrc{\mathrm{rrc}}

\newcommand\x{\mathbf{x}}

\DeclareMathOperator{\conv}{conv}
\DeclareMathOperator{\cone}{cone}
\DeclareMathOperator{\aff}{aff}
\DeclareMathOperator{\homog}{hom}

\newcommand\rk{\mathrm{rk}}
\newcommand\rkN{\rk_+}
\newcommand\rkPSD{\rk_\mathrm{psd}}

\newcommand\Gnk[2]{\mathcal{G}({#1,#2})}

\newcommand\Rel{\mathcal{R}}

\newcommand{\ffloor}[2]{\left\lfloor{\frac{#1}{#2}}\right\rfloor} %
\newcommand{\fceil}[2]{\left\lceil {\frac{#1}{#2}} \right\rceil} %

\newcommand\ext[1]{\widehat{{#1}}}

\newtheorem{thm}{Theorem}[section]
\newtheorem{cor}[thm]{Corollary}
\newtheorem{prop}[thm]{Proposition}
\newtheorem{lem}[thm]{Lemma}
\newtheorem{conj}{Conjecture}
\newtheorem{quest}{Question}

\theoremstyle{definition}

\newtheorem{ex}{Example}

\makeatletter
\newtheorem*{rep@theorem}{\rep@title}
\newcommand{\newreptheorem}[2]{%
\newenvironment{rep#1}[1]{%
 \def\rep@title{#2 \ref{##1}}%
 \begin{rep@theorem}}%
 {\end{rep@theorem}}}
\makeatother

\newreptheorem{thm}{Theorem}
\newreptheorem{cor}{Corollary}
\newreptheorem{conj}{Conjecture}

\title{Extension complexity and realization spaces of hypersimplices}

\author{Francesco Grande}
\address{Fachbereich Mathematik und Informatik, %
Freie Universit\"at Berlin, Berlin, %
Germany}
\email{fgrande@math.fu-berlin.de}

\author{Arnau Padrol}
\address{
Institut de Mathématiques de Jussieu - Paris Rive Gauche (UMR 7586),
Universit\'e Pierre et Marie Curie (Paris 6), Paris, %
France}
\email{arnau.padrol@imj-prg.fr}

\author{Raman Sanyal}
\address{Institut f\"ur Mathematik, Goethe-Universit\"at Frankfurt, Germany}
\email{sanyal@math.uni-frankfurt.de}


\date{\today}
\thanks{F.~Grande was supported by DFG within the research training group
``Methods for Discrete Structures'' (GRK1408).  A.~Padrol and R.~Sanyal were
supported by the DFG Collaborative Research Center SFB/TR 109 ``Discretization
in Geometry and Dynamics''. 
A.~Padrol was also supported by the program PEPS
Jeunes Chercheur-e-s 2016 of the INSMI (CNRS)
}

\begin{document}

\maketitle

\begin{abstract}
    The $(n,k)$-hypersimplex is the convex hull of all $0/1$-vectors of length
    $n$ with coordinate sum $k$. We explicitly determine the extension
    complexity of all hypersimplices as well as of certain classes of
    combinatorial hypersimplices.  To that end, we investigate the projective
    realization spaces of hypersimplices and their (refined) rectangle
    covering numbers. Our proofs combine ideas from geometry and
    combinatorics and are partly computer assisted.
\end{abstract}

\section{Introduction}\label{sec:intro}

\subsection{The extension complexity of hypersimplices}

The \Defn{extension complexity} or \Defn{nonnegative rank} $\rkN(P)$ of a
convex polytope $P$ is the minimal number of facets (i.e., describing linear
inequalities) of an extension, a polytope $\ext{P}$ that linearly projects
onto $P$. The motivation for this definition comes from linear optimization:
The computational complexity of the simplex algorithm is intimately tied to
the number of linear inequalities and hence it can be advantageous to optimize
over~$\ext{P}$.  As a complexity measure, the nonnegative rank is an object of
active research in combinatorial optimization; see~\cite{dagstuhl}.  There are
very few families of polytopes for which the exact nonnegative rank is known.
Besides simplices, examples are cubes, crosspolytopes, Birkhoff polytopes and
bipartite matching polytopes~\cite{FKPT13} as well as  all $d$-dimensional
polytopes with at most $d+4$ vertices~\cite{Padrol16}.
Determining the nonnegative rank is non-trivial even for
polygons~\cite{polygons1,polygons2, shitov2, shitov}. For important classes of
polytopes exponential lower bounds obtained in~\cite{FMPT,rothvoss,rothvoss2}
are celebrated results.  

In the first part of the paper we explicitly
determine the nonnegative rank of the family of hypersimplices. For $0 < k <
n$, the \Defn{$\boldsymbol{(n,k)}$-hypersimplex} is the convex polytope
\begin{equation}\label{eqn:hyper}
    \Delta_{n,k} \ = \ \conv\left\{ x \in \{0,1\}^n : x_1 + \cdots + x_n =
    k\right\}.
\end{equation}
Hypersimplices were first described (and named) in connection with moment
polytopes of orbit closures in Grassmannians (see~\cite{GGMS}) but, of course,
they are prominent objects in combinatorial optimization, appearing in
connection with packing problems and matroid theory; see also below. 
This marks hypersimplices as polytopes of considerable interest and naturally
prompts the question as to their extension complexity.  

Note that
$\Delta_{n,k}$ is affinely isomorphic to $\Delta_{n,n-k}$. The hypersimplex
$\Delta_{n,1} = \Delta_{n-1}$ is the standard simplex of dimension $n-1$ and
$\rkN(\Delta_{n-1}) = n$. Our first result concerns the extension complexity
of the \emph{proper} hypersimplices, that is, the hypersimplices $\Delta_{n,k}$
with $2 \le k \le n-2$.

\begin{thm}\label{thm:main}
    The hypersimplex $\Delta_{4,2}$ has extension complexity $6$, the
    hypersimplices $\Delta_{5,2} \cong \Delta_{5,3}$ have extension complexity
    $9$. For any $n\ge 6$ and $2 \le k \le n-2$, we have $\rkN(\Delta_{n,k}) =
    2n$.
\end{thm}

It is straightforward to check that
\begin{equation}\label{eqn:ineq}
    \Delta_{n,k} \ = \  [0,1]^n \cap \{ x \in \R^n :  x_1 + \cdots + x_n = k\}
\end{equation}
and that for $1 < k < n-1$, all $2n$ inequalities of the $n$-dimensional cube
are necessary. The nonnegative rank of a polytope is trivially upper bounded
by the minimum of the number of vertices and the number of facets. We call a
polytope $P$ \Defn{extension maximal} if it attains this upper bound. Cubes as
well as their duals, the crosspolytopes, are know to be extension maximal; see
also Corollary~\ref{cor:cube}. Theorem~\ref{thm:main} states that in addition
to simplices, cubes, and crosspolytopes, all proper hypersimplices except for
$\Delta_{5,2}$ are extension maximal. 

\subsection{Psd rank and $2$-level matroids}
Our original motivation for studying the nonnegative rank of hypersimplices
comes from matroid theory~\cite{Oxley}. For a matroid $M$ on the ground set $[n]
:= \{1,\dots,n\}$ and bases $\mathcal{B} \subseteq 2^{[n]}$, the associated
\Defn{matroid base polytope} is the polytope
\[
    P_M \ := \ \conv \{ \1_B : B \in \mathcal{B}\},
\]
where $\1_B \in \{0,1\}^n$ is the characteristic vector of $B \subseteq [n]$.
Hence, the $(n,k)$-hypersimplex is the matroid base polytope of the uniform
matroid $U_{n,k}$. In~\cite{gs14}, the first and third author studied
\mbox{\Defn{$\boldsymbol 2$-level matroids}}, which exhibit extremal behavior
with respect to various geometric and algebraic measures of complexity. In
particular, it is shown that $M$ is $2$-level if and only if $P_M$ is
\Defn{psd minimal}. The \Defn{psd rank} $\rkPSD(P)$ of a polytope $P$ is the
smallest size of a spectrahedron (an affine section of the positive definite
cone) that projects onto $P$. In~\cite{GRT} it is shown that $\rkPSD(P) \ge
\dim P + 1$ and polytopes attaining this bound are called psd minimal. Our
starting point was the natural question whether the class of $2$-level
matroids also exhibits an extremal behavior with respect to the nonnegative
rank.  We recall from~\cite[Theorem~1.2]{gs14} the following synthetic
description of $2$-level matroids: A matroid $M$ is $2$-level if and only if
it can be constructed from uniform matroids by taking direct sums or $2$-sums.
So, the right starting point are the hypersimplices.  

To extend Theorem~\ref{thm:main} to all $2$-level matroids, it would be
necessary to understand the effect of taking direct and $2$-sums on the
nonnegative rank.  The direct sum of matroids translates into the Cartesian
product of matroid polytopes. Two out of three authors of this paper believe
in the following conjecture, first asked during a Dagstuhl seminar in
2013~\cite{dagstuhlseminar2013}. 

\begin{conj}\label{conj:prod}
    The nonnegative rank is additive with respect to Cartesian products, that
    is,
    \[
        \rkN(P_1 \times P_2) \ = \ \rkN(P_1) + \rkN(P_2),
    \]
    for polytopes $P_1$ and $P_2$.
\end{conj}

We provide evidence in favor of Conjecture~\ref{conj:prod} by showing it to
hold whenever one of the factors is a simplex
(cf.~Corollary~\ref{cor:prod_simplex}).
By taking products of extensions it trivially follows that the nonnegative
rank is subadditive with respect to Cartesian products. As for the $2$-sum
$M_1 \oplus_2 M_2$ of two matroids $M_1$ and $M_2$, it follows
from~\cite[Lemma~3.4]{gs14} that $P_{M_1 \oplus_2 M_2}$ is a codimension-$1$
section of $P_{M_1} \times P_{M_2}$ and the extension complexity is therefore
dominated by that of the direct sum. Combined with Theorem~\ref{thm:main}
and~\cite[Theorem~1.2]{gs14} we obtain the following simple estimate.

\begin{cor}
    If $M$ is a $2$-level matroid on $n$ elements, then $\rk(P_M) \le 2n$.
\end{cor}

\subsection{Extension complexity of combinatorial hypersimplices}
The extension complexity is not an invariant of the combinatorial type. That
is, two combinatorially isomorphic polytopes do not necessarily have the same
extension complexity. For example, the extension complexity of a hexagon is
either $5$ or $6$ depending on the incidences of the facet-defining
lines~\cite[Prop.~4]{polygons2}. On the other hand, the extension complexity
of any polytope combinatorially isomorphic to the $n$-dimensional cube is
always $2n$; cf.\ Corollary~\ref{cor:cube}. The close connection to simplices
and cubes and Theorem~\ref{thm:main} raises the following question for
\Defn{combinatorial} $(n,k)$-hypersimplices.

\begin{quest}\label{quest:comb}
    Is $\rkN(P)=2n$ for any combinatorial $(n,k)$-hypersimplex $P$ with $n\geq
    6$ and $2\leq k\leq n-2$?
\end{quest}
For $n=6$ and $k\in\{ 2,3\}$ this is true due to
Proposition~\ref{prop:small_cases} but we suspect that the answer is no for
some $n > 6$ and $k=2,n-2$. The rectangle covering number $\rc(P)$ of a
polytope~$P$ is a combinatorial invariant that gives a lower bound on
$\rkN(P)$; see Section~\ref{sec:rc}.  While the rectangle covering number of
the \emph{small} hypersimplices $\Delta_{6,2}$ and $\Delta_{6,3}$ is key to
our proof of Theorem~\ref{thm:main}, it is not strong enough to resolve
Question~\ref{quest:comb} (see Proposition~\ref{prop:rcbounds}).

We introduce the notions of $F$-, $G$-, and $FG$-genericity of combinatorial
hypersimplices, that are defined in terms of the relative position of certain
facets and that play a crucial role. We show that all $FG$-generic
hypersimplices are extension maximal (Theorem~\ref{thm:main_generic}).
Unfortunately, $FG$-genericity is not a property met by all hypersimplices,
which is confirmed by the existence of a non-$FG$-generic realization of
$\Delta_{6,2}$; see Proposition~\ref{prop:singular62}.  On the other hand, we
show that hypersimplices with $n \ge 6$ and $\lfloor\frac{n}{2}\rfloor \le k
\le \lceil\frac{n}{2}\rceil$ are $FG$-generic, which ensues the following.

\begin{cor}\label{cor:XCcombinatorial}
    If $P$ is a combinatorial $(n,k)$-hypersimplex with $n\geq 6$ and $2\leq
    k\leq \lceil\frac{n}{2}\rceil$, then
    \[ 
        \rkN(P) \ \geq \ 
        \begin{cases}
            n+2k+1 & \text{ if }k<\ffloor{n}{2},\\
            2n & \text{ otherwise.}
              \end{cases}
              \]
\end{cor}

We do not know of any realization of a $(n,k)$-hypersimplex with $n \ge 6$ of
extension complexity less than $2n$, but we do not dare to conjecture that
every combinatorial $(n,k)$-hypersimplex with $n \ge 6$ and $2 \le k \le n$ is
extension maximal.

\subsection{Realization spaces of hypersimplices}
The \Defn{projective realization space} $\Rel_{n,k}$ of combinatorial
$(n,k)$-hypersimplices parametrizes the polytopes combinatorially isomorphic
to $\Delta_{n,k}$ up to projective transformation. (Projective) realization
spaces of polytopes are provably complicated objects. The \emph{universality
theorems} of Mn\"ev~\cite{Mnev1988} and
Richter-Gebert~\cite{RichterGebert1997} assert that realization spaces of
polytopes of dimension $\ge 4$ are as complicated as basic open semialgebraic
sets defined over the integers. 
In contrast, for a $3$-dimensional polytope~$P$ with $e \ge 9$ edges, it
follows from Steinitz' theorem that the projective realization space is
homeomorphic to an open ball of dimension $e - 9$; see
also~\cite[Thm.~13.3.3]{RichterGebert1997}. 

For our investigation of the extension complexity of combinatorial
hypersimplices, we study their realization spaces.  The observation that every
hypersimplex is either $F$- or $G$-generic (Lemma~\ref{lem:generic}) turns out
to be instrumental in our study. For $k=2$, we are able to give a full
description.

\begin{thm}\label{thm:dim_rel_space}
    For $n \geq 4$, $\Rel_{n,2}$ is {rationally} equivalent to the
    interior of a $\binom{n-1}{2}$-dimensional cube. In particular,
    $\Rel_{n,2}$ is homeomorphic to an open ball and hence contractible. 
\end{thm}

Rationally equivalent means that the homeomorphism as well as its inverse are
given by rational functions (c.f.~\cite[Sect.~2.5]{RichterGebert1997}).

A key tool in the context of the Universality Theorem is that the
projective realization of a facet of a high-dimensional polytope can not be
prescribed in general; see, for example,~\cite[Sect.~6.5]{Ziegler1995}. In contrast,
the shape of any single facet of a $3$-polytope can be prescribed~\cite{BarnetteGruenbaum1970}.
This description of $\Rel_{n,2}$ allows us to show that facets
of $(n,2)$-hypersimplices can be prescribed (Corollary~\ref{cor:n2prescribability}), but also allows us to
construct hypersimplices that are not
$FG$-generic, which implies that facets of hypersimplices cannot be prescribed
 in general (Corollary~\ref{cor:prescribability}).

For $2 < k < n-2$, the realization spaces are more involved and, in
particular, related to the algebraic variety of $n$-by-$n$ matrices with
vanishing principal $k$-minors that was studied by Wheeler~\cite{Wheeler15}.
In Theorem~\ref{thm:Rel_UB}, we show that certain facets of $\Delta_{n,k}$
completely determine the realization, which then gives an upper bound on the
dimension of the realization space.  However, we can currently not exclude
that $\Rel_{n,k}$ is disconnected and has components of different dimensions.

The extension complexity is invariant under (admissible) projective
transformations and hence $\rkN$ is well-defined on $\Rel_{n,k}$. The locus
$E_{n,k} \subseteq \Rel_{n,k}$ of extension maximal $(n,k)$-hypersimplices is
open and Theorem~\ref{thm:main} implies that $E_{n,k}$ is non-empty for $n \ge
6$ and $2 \le k \le n-2$. For $k=2$, we can say considerably more.

\begin{cor}\label{cor:dense2}
    For $n \ge 5$, the combinatorial $(n,2)$-hypersimplices with extension
    complexity $2n$ are dense in $\Rel_{n,2}$.
\end{cor}

Our results on $FG$-generic hypersimplices, which are characterized by the
non-vanishing of a determinantal condition on $\Rel_{n,k}$, strongly suggest
that Corollary~\ref{cor:dense2} extends to all the cases.

\begin{conj}\label{conj:dense}
    For $n\geq 5$ and $2\leq k\leq n-2$, the combinatorial hypersimplices of
    nonnegative rank~$2n$ form a dense open subset of $\Rel_{n,k}$.
\end{conj}

\subsection{Structure of the paper}
Theorem~\ref{thm:main} is proved in Sections~\ref{sec:geom} and~\ref{sec:rc}.
In Section~\ref{sec:geom} we investigate the discrete geometry of extensions
and we set up an induction that deals with the \emph{large} hypersimplices
$\Delta_{n,k}$ with $n > 6$. In particular, we devise general tools for upper
bounding the extension complexity. For the \emph{small} hypersimplices
$\Delta_{6,2}$ and $\Delta_{6,3}$, we make use of rectangle covering numbers
in Section~\ref{sec:rc}. We show that most of the geometric tools of
Section~\ref{sec:geom} have combinatorial counterparts for rectangle covering
numbers. Section~\ref{sec:relspaces} is devoted to the study of combinatorial
hypersimplices and the associated realization spaces. In Section~\ref{sec:n2}
we focus on the combinatorial $(n,2)$-hypersimplices.

\section{The geometry of extensions and large hypersimplices}\label{sec:geom}

In this section we develop some useful tools pertaining to the geometry of
extensions.  These will be used to give an inductive argument for the
\emph{large} hypersimplices $\Delta_{n,k}$ with $n > 6$ and $1 < k < n-1$. The
\emph{small} hypersimplices are treated in the next section.

For a polytope~$P$, we write~$v(P)$ for the number of vertices of~$P$
and~$f(P)$ for the number of facets. Moreover, $\ext{P}$ will typically denote
an extension of~$P$, and the linear projection that takes~$\ext{P}$ to~$P$
is denoted by~$\pi$. We start with the simple observation that the nonnegative rank is
strictly monotone with respect to taking faces.

\begin{lem}\label{lem:facet}
    Let $P$ be a polytope and $F \subset P$ a facet. Then
    \[
        \rkN(P) \ \ge \ \rkN(F) + 1.
    \]
\end{lem}
\begin{proof}
    Let $\ext{P}$ be a minimal extension of $P$. The preimage $\ext{F} =
    \pi^{-1}(F) \cap \ext{P}$ is an extension of~$F$. Every facet of $\ext{F}$
    is the intersection of a facet of $\ext{P}$ with $\ext{F}$. Moreover,
    since $\ext{F}$ is a proper face of $\ext{P}$, there are at least $c \ge
    1$ facets of $\ext{P}$ that contain $\ext{F}$ and hence do not contribute
    facets to $\ext{F}$. It follows that
    \[
        \rkN(P)  \ = \ f(\ext{P}) \ \ge \ f(\ext{F}) + c \ \ge \ \rkN(F) + 1,
    \]
    which proves the claim.
\end{proof}

By induction, this extends to lower dimensional faces.
\begin{cor}\label{cor:face}
    Let $P$ be a polytope and $F \subset P$ a face. Then
    \[
        \rkN(P) \ \ge \ \rkN(F) + \dim(P) - \dim(F).
    \]
\end{cor}

We can strengthen this observation if we take into consideration more than one
facet.

\begin{lem}\label{lem:2distinct}
    Let $P$ be a polytope and let $F_1$ and $F_2$ be two disjoint facets of
    $P$. Then
    \[
        \rkN(P) \ \ge \ \min\left\{\rkN(F_1),\rkN(F_2)\right\} + 2.
    \]
\end{lem}
\begin{proof}
    If $\rkN(F_1) > \rkN(F_2)$, the claim follows from Lemma~\ref{lem:facet}.
    Hence, we can assume that $\rkN(F_1) = \rkN(F_2) = k$.  Extending the
    argument of Lemma~\ref{lem:facet}, let $\ext{P}$ be a minimal extension of
    $P$ and $\ext{F}_i$ the preimage of $F_i$ for $i=1,2$. Let $c_i$ be the
    number of facets of $\ext{P}$ containing $\ext{F}_i$. Since $f(\ext{P}) \ge
    k + c_i$, the relevant case is $c_1 = c_2 = 1$.  Now, $\pi(\ext{F}_1
    \cap \ext{F}_2) \subseteq F_1 \cap F_2 = \emptyset$ implies that
    $\ext{F}_1$ and $\ext{F}_2$ are disjoint facets of $\ext{P}$. Hence,
    $\rkN(P)  = f(\ext{P}) \ \ge \ k + 2$.
\end{proof}

We cannot replace $\min$ with $\max$ in Lemma~\ref{lem:2distinct}: The convex
hull of the $12$ columns of the matrix
\[
    \left(
    \begin{array}{rrrrrrrrrrrr}
        1 & -1 & -1 &  1 &  2 &  2 & -2 & -2 &  1 & -1 &  1 & -1 \\
        2 &  2 & -2 & -2 &  1 & -1 & -1 &  1 &  1 &  1 & -1 & -1 \\
        1 &  1 &  1 &  1 &  1 &  1 &  1 &  1 & -1 & -1 & -1 & -1 \\
        \hline
        1 &  1 &  1 &  1 & -1 & -1 & -1 & -1 & -1 & -1 & -1 & -1
    \end{array}
    \right)
\]
gives rise to a $4$-dimensional polytope $Q$ combinatorially isomorphic to
product of a triangle and a quadrilateral, and consequently has $7$ facets. 
If we project onto the first three coordinates we obtain a $3$-dimensional
polytope $P$ with two parallel facets, $F_1$ and $F_2$, that are an octagon
and a square, with nonnegative ranks $6$ and $4$, respectively. Thus,
$\rkN(P) \le 7 < \max\{\rkN(F_1),\rkN(F_2)\} + 2 = 8$. Figure~\ref{fig:sqoct}
gives an idea of the geometry underlying $Q$.

\begin{figure}[htbp]
\includegraphics[width=.7\linewidth]{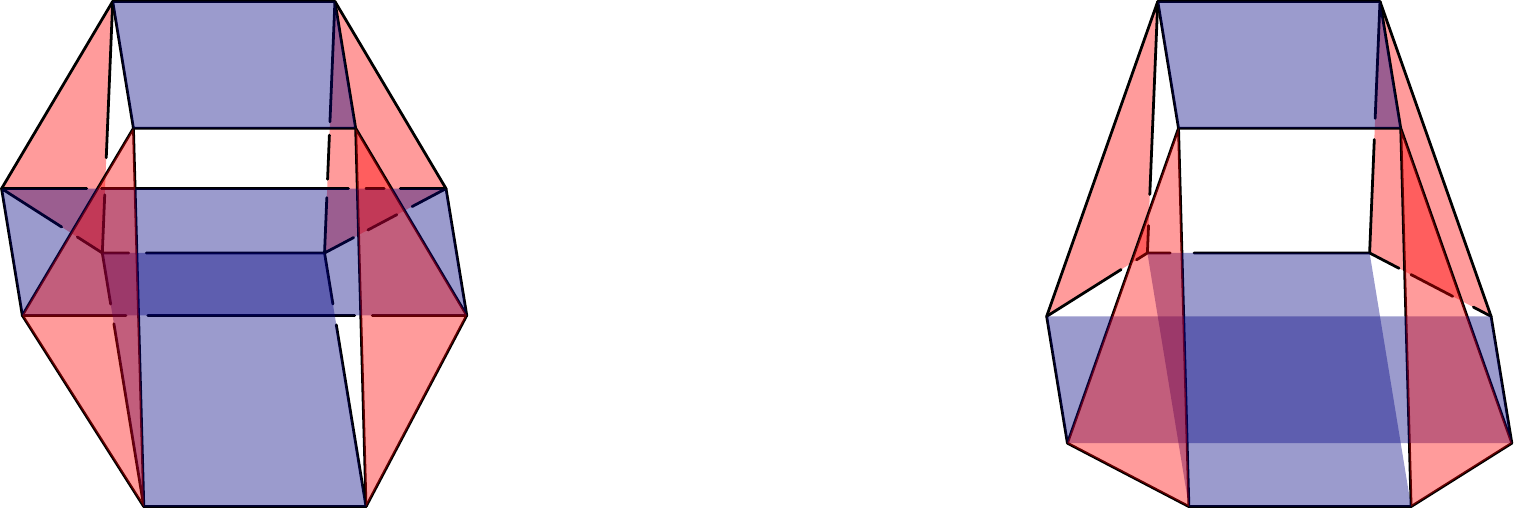}
\caption{ The left figures gives a sketch (not a Schlegel diagram) of the
geometric idea underlying the construction of $Q$. It is a union of three
facets that yield the projection on the right.  We highlighted the structure
as a product of polygons, that makes it more visible how the two square faces
of $Q$ yield the octogonal face of $P$.}
\label{fig:sqoct}
\end{figure}

Combining Lemma~\ref{lem:facet} and Lemma~\ref{lem:2distinct} yields the
following result pertaining to Conjecture~\ref{conj:prod}.
\begin{cor}\label{cor:prod_simplex}
    Let $P$ be a non-empty convex polytope and $k \ge 1$. Then
    \[
        \rkN(P \times \Delta_{k}) \ = \ \rkN(P) + k + 1.
    \]
\end{cor}
\begin{proof}
    Let $\ext{P}$ be a minimal extension of $P$ with $\rkN(P)$ facets. Since
    the number of facets of a product add up, $\ext{P} \times \Delta_k$ is an
    extension of $P \times \Delta_k$ with $\rkN(P)+k+1$ facets.  Thus, we need
    to show that $\rkN(P)+k+1$ is also a lower bound.

    For $k=1$, the polytope $P \times \Delta_1$ is a prism over $P$ with two
    distinct facets isomorphic to $P$ and the claim follows from
    Lemma~\ref{lem:2distinct}. If $k > 1$, note that $P \times \Delta_{k-1}$
    is a facet of $P \times \Delta_{k}$, and an application of Lemma~\ref{lem:facet}
    yields the claim by induction on $k$.
\end{proof}

Another byproduct is a simple proof that every combinatorial cube is extension
maximal (see \cite[Proposition 5.9]{FKPT13}).

\begin{cor}\label{cor:cube}
    If $P$ is combinatorially equivalent to the $n$-dimensional cube $C_n =
    [0,1]^n$, then $\rkN(P) = 2n$.
\end{cor}
\begin{proof}
    Since $f(P) = f(C_n) = 2n$, we only need to prove $\rkN(P) \ge 2n$.  For
    $n=1$, $P$ is a $1$-dimensional simplex for which the claim is true. For
    $n \ge 2$ observe that $P$ has two disjoint facets $F_1,F_2$ that are
    combinatorially equivalent to $(n-1)$-cubes. By induction and
    Lemma~\ref{lem:2distinct} we compute $\rkN(P)  \ge  \rkN(C_{n-1}) + 2  =
    2n$.
\end{proof}

With these tools, we are ready to prove Theorem~\ref{thm:main} for the cases
with $n>6$. The case $n=6$ and $1 < k < n-1$ will be treated in
Proposition~\ref{prop:small_cases} in the next section. A key property,
inherited from cubes, that allows for an inductive treatment of hypersimplices
is that for $1 < k < n-1$, the presentation~\eqref{eqn:ineq} purports that 
\begin{equation}\label{eqn:FG}
\begin{aligned}
    F_i \ &:= \ \Delta_{n,k} \cap \{x_i = 0\} \ \cong \ \Delta_{n-1,k}, 
    \text{ and } \\
    G_i \ &:= \ \Delta_{n,k} \cap \{x_i = 1\} \ \cong \ \Delta_{n-1,k-1}, 
\end{aligned}
\end{equation}
are disjoint facets for any $1 \le i \le n$. We call these the
\Defn{$\boldsymbol F$-facets} and  \Defn{$\boldsymbol G$-facets},
respectively.

\begin{prop}\label{prop:large_cases}
    Assume that $\rkN(\Delta_{6,2}) = \rkN(\Delta_{6,3}) = 12$. Then
    $\rkN(\Delta_{n,k}) = 2n$ for all $n > 6$ and $1 < k < n-1$.
\end{prop}

\begin{proof}
    Let $n \ge 7$.  For $2 < k < n-2$, the pairs of disjoint
    facets~\eqref{eqn:FG} allow us to use Lemma~\ref{lem:2distinct} together
    with induction on $n$ and $k$ to establish the result. Hence, the relevant
    cases are $n \ge 7$ and $k = 2$ (which is equivalent to $k=n-2$).
    
    For $k = 2$, let
    \[
        \ext{P} \ = \ \{ y \in \R^m : \ell_i(y) \ge 0  \text{ for }
        i=1,\dots,M \}
    \]
    be an extension of $\Delta_{n,k}$ given by affine linear forms
    $\ell_1,\dots,\ell_M$ and $M = \rkN(\Delta_{n,k})$. For convenience, we
    can regard $\Delta_{n,k}$ as a full-dimensional polytope in the affine
    hyperplane \mbox{$\{ x\in\R^n : x_1 + \cdots + x_n = k \} \cong \R^{n-1}$.
    Let $\pi : \R^m \rightarrow \R^{n-1}$} the linear projection that takes
    $\ext{P}$ to $\Delta_{n,k}$. If for some $1 \le i \le n$, the preimage
    $\ext{F}_i = \pi^{-1}(F_i) \cap \ext{P}$ is not a facet then $f(\ext{P})
    \ge \rkN(F_i) + 2 = 2n$ by induction and we are done. So, we have to
    assume that $\ext{F}_i = \{ y \in \ext{P} : \ell_i(y) = 0 \}$ is a facet
    of $\ext{P}$ for all $i=1,\dots,n$. 

    It is sufficient to show that the polyhedron $\ext{Q} := \{ y \in \R^m :
    \ell_i(y) \ge 0 \text{ for } i = n+1,\dots,M \}$ is bounded and hence has
    $f(\ext{Q}) \ge m+1 \ge n$ facets. Since $f(\ext{P}) = n + f(\ext{Q})$
    this implies the result.  The key observation is that the polyhedron $Q
    \subset \R^{n-1}$ bounded by the hyperplanes defining the facets $G_i$ of
    $\Delta_{n,k}$ is a full-dimensional simplex and hence bounded.  We claim
    that $\pi(\ext{Q}) \subseteq Q$.  For this it is sufficient to show that
    if $H_i$ is the unique hyperplane containing $G_i$, then $\pi^{-1}(H_i)$
    supports a face of $\ext{Q}$. By construction, $\pi^{-1}(H_i)$ supports
    the face $\ext{G}_i := \pi^{-1}(G_i) \cap \ext{P}$ of $\ext{P}$.  Now, if
    $\ext{G}_i \subseteq \ext{F}_j$ for some $1 \le j \le n$, this would imply
    $G_i \subseteq F_j$. This, however, cannot happen as $G_i =
    \pi(\ext{G}_i)$ and $F_j$ are distinct facets of $\Delta_{n,k}$.  Thus,
    $\ext{G}_i = \{ y \in \ext{P} : \ell_j(y) = 0 \text{ for } j \in J\}$ for
    some $J \subseteq \{ n+1,\dots,M\}$ and consequently $H_i$ supports a face
    of $\ext{Q}$.  Moreover, $\ext{Q} \subseteq \pi^{-1}(Q)$ and hence, the
    lineality space of $\ext{Q}$ is contained in $\ker \pi$. However the
    hyperplanes $\{ y : \ell_i(y) = 0 \}$ with $1\leq i\leq n$ are parallel to
    $\ker \pi$, because they are preimages of the hyperplanes supporting the
    facets $F_i$.  Therefore, $\ext{Q}$ is bounded since we assumed that
    $\ext{P} = \ext{Q} \cap \{ y : \ell_i(y) \ge 0 \text{ for } i =
    1,\dots,n\}$ is bounded.
\end{proof}

\section{Rectangle covering numbers and small hypersimplices}\label{sec:rc}

In this section we treat the small hypersimplices $4 \le n \le 6$ and $1 < k <
n-1$. We will do this by way of rectangle covering numbers.  The rectangle
covering number, introduced in~\cite{FKPT13}, is a very elegant, combinatorial
approach to lower bounds on the nonnegative rank of a polytope.  For a
polytope 
\[
    P \ = \ \{ x \in \R^d : \ell_1(x) \ge 0, \dots,\ell_M(x) \ge 0\} \ = \
    \conv(v_1,\dots,v_N),
\]
where $\ell_1(x),\dots,\ell_M(x)$ are affine linear forms,
the \Defn{slack matrix} is the nonnegative matrix $S_P \in \R_{\ge0}^{M \times
N}$  with $(S_P)_{ij} = \ell_i(v_j)$.  A \Defn{rectangle} of $S_P$ is an index
set $R = I \times J$ with $I \subseteq [M]$, $J \subseteq [N]$ such that
$(S_P)_{ij} > 0$ for all $(i,j) \in R$. The \Defn{rectangle covering number}
$\rc(S_P)$ is the smallest number of rectangles $R_1,\dots,R_s$ such that
$(S_P)_{ij} > 0$ if and only if $(i,j) \in \bigcup_t R_t$. As explained
in~\cite[Section~2.4]{FKPT13}
\[
    \rc(S_P) \ \le \  \rkN(P).
\]

There are strong ties between the geometry of extensions and rectangle
covering numbers. In particular our geometric tools from
Section~\ref{sec:geom} have independent counterparts for rectangle covering
numbers. Note that although the results are structurally similar they do not
imply each other and even the proofs are distinct.

\begin{lem}\label{lem:rc}
    Let $P$ be a polytope and $F \subset P$ a facet. Then
    \[
        \rc(S_P) \ \ge \ \rc(S_F) + 1.
    \]
    Moreover, if there is a facet $G \subset P$ disjoint from $F$, then 
    \[
        \rc(S_P) \ \ge \ \min\{\rc(S_F),\rc(S_G)\} + 2.
    \]
\end{lem}
\begin{proof}
    In the first case, part of the slack matrix of $S_P$ is of the form
    \[
        \renewcommand\arraystretch{.85}
        \left[ \begin{array}{c@{\,}c@{\,}c|c}
                 0& \cdots & 0 & a \\
                \hline
                & &          & \ast \\
                & S_F &          & \vdotsHacked \\
                &  &          & \ast \\
        \end{array} \right],
    \]
    and since $F$ is a facet, $a > 0$. There are at least $\rc(S_F)$
    rectangles necessary to cover $S_F$. None of these rectangles can cover
    $a$ as this is obstructed by the zero row above $S_F$.

    For the second case, we may assume that $r = \rc(S_F) = \rc(S_G)$.
    Similarly, we can assume that parts of $S_P$ look like 
    \[
       \renewcommand\arraystretch{.85}
        \left[
        \begin{array}{c@{\,}c@{\,}c|c@{\,}c@{\,}c}
                0& \cdots & 0 & a_1 &  \cdots &  a_l \\
                b_1 & \cdots & b_k & 0 &  \cdots &  0 \\
                \hline
                &  &          &   \ast & \cdots & \ast        \\ 
                & S_F &          &  \vdotsHacked &    &
                \vdotsHacked        \\
                &  &          & \ast  & \cdots  & \ast        \\
		\hline
               \ast & \cdots & \ast         &   &  &         \\   
                \vdotsHacked & & \vdotsHacked        &   & S_G &         \\
                \ast  & \cdots  & \ast         &   &  &         \\
        \end{array}
    \right]
    \]
    with $a_1,\dots,a_l,b_1,\dots,b_k > 0$. There are $r$ rectangles necessary
    to cover $S_F$. None of these rectangles can cover the first row. If the
    first row is covered with $\ge 2$ rectangles, we are done. If, however, a
    single rectangle covers the first row, then it cannot cover any
    row of $S_G$. Indeed, every row of $S_G$ corresponds to a facet of $G$ and
    contains at least one vertex of $G$. Hence, every row of $S_G$ has one
    zero entry. Since also $S_G$ needs at least 
    $r$ rectangles to be covered, by the
    same token we obtain that the second row must be covered by a
    unique rectangle which does not extend to $S_F$ or $S_G$. Consequently,
    at least $r+2$ rectangles are necessary.
\end{proof}

The example from Section~\ref{sec:geom} shows that similar to
Lemma~\ref{lem:2distinct}, we cannot replace $\min$ with $\max$. It can be
checked that the rectangle covering number of an octagon is $6$.

As direct consequence we obtain a lower bound on rectangle covering numbers (cf.~\cite[Prop.~5.2]{FKPT13}).

\begin{cor}\label{cor:lb_rc}
    Let $P$ be a $d$-polytope, then $\rc(S_P)\geq d+1$.
\end{cor}

It was amply demonstrated in~\cite{KaibelW15,FMPT} that the rectangle covering
number is a very powerful tool.  We use it to compute the nonnegative rank of
small hypersimplices. For a given polytope $P$ with slack matrix $S = S_P$ the
decision problem of whether there is a rectangle covering with $r$ rectangles
can be phrased as a satisfiability problem: For every rectangle $R_l$ and
every $(i,j)$ with $S_{ij} >0$ we designate a Boolean variable $X^l_{ij}$. If
$X^l_{ij}$ is true, this signifies that $(i,j) \in R_l$. Every $(i,j)$ has to
occur in at least one rectangle. Moreover, for $(i,j)$ and $(i',j')$ if
$S_{ij} \cdot S_{i'j'} > 0$ and $S_{ij'} \cdot S_{i'j} = 0$, then $(i,j)$ and
$(i',j')$ cannot be in the same rectangle. The validity of the resulting
Boolean formula can then be verified using a SAT solver. For the
hypersimplices $\Delta_{n,k}$ with $1<k<n-1$, the sizes of the slack matrix is
$2n \times \binom{n}{k}$.  For $n \le 6$ these sizes are manageable and the
satisfiability problem outlined above can be decided by a computer.  For
example, for $(n,k) = (6,3)$ this yields $1320$ Boolean variables and $55566$
clauses in a conjunctive normal form presentation. The attached
\texttt{python} script produces a SAT instance for all $(n,k,r)$ and we used
\texttt{lingeling}~\cite{lingeling} for the verification. This gives a
computer-aided proof for the small cases which also completes our proof for
Theorem~\ref{thm:main}.

\begin{prop}\label{prop:small_cases}
    For $n \le 6$, $\rc(\Delta_{n,k}) = \rkN(\Delta_{n,k})$ for all $1 \le k
    \le n$. In particular, $\rkN(\Delta_{4,2})=6$,
    $\rkN(\Delta_{5,2})=\rkN(\Delta_{5,3})=9$, and
    $\rkN(\Delta_{6,2})=\rkN(\Delta_{6,3})=12$.
\end{prop}
\begin{proof}
    The hypersimplex $\Delta_{4,2}$ is a $3$-dimensional polytope with $6$
    vertices and, more precisely, affinely isomorphic to the octahedron. Since
    the nonnegative rank is invariant under taking polars,
    Corollary~\ref{cor:cube} asserts that the nonnegative rank is indeed $6$.
    The polytope $\Delta_{5,2}$ is a $4$-dimensional polytope with $10$
    vertices and facets.  Its nonnegative rank is $9$. It was computed
    in~\cite[Table 3]{OVW14} under the ID 6014. Alternatively it can be
    computed with the \texttt{python} script in the appendix. Using, for
    example, {\tt polymake}~\cite{polymake}, removing two non-adjacent
    vertices of $\Delta_{5,2}$ yields a $4$-dimensional polytope $Q$ with $8$
    vertices and $7$ facets. Taking a $2$-fold pyramid over $Q$ gives an
    extension of $\Delta_{5,2}$ with $9$ facets. Finally, $\Delta_{6,2}$ and
    $\Delta_{6,3}$ are $5$-polytopes with $12$ facets and the SAT approach
    using the attached \textit{python} script yields the matching lower bound
    on the rectangle covering number.
\end{proof}

The hypersimplex $\Delta_{5,2}$ is special. We will examine it more closely in 
Section~\ref{sec:52} and we will, in particular, show that up to a set of
measure zero all realizations have the expected nonnegative rank $10$.
  
It is tempting to think that Proposition~\ref{prop:large_cases} might hold on
the level of rectangle covering numbers. Indeed, such a result would imply
that all \emph{combinatorial} hypersimplices are extension maximal.  As can be
checked with the \texttt{python} script in the appendix,
Proposition~\ref{prop:small_cases} extends at least to $n=8$.  In fact, the
results above imply that $\rc(\Delta_{n,k})=2n$ when $\max\{2,n-6\}\leq k\leq
\min\{n-2,6\}$.  However, the same script also shows that
$\rc(\Delta_{10,2})\leq 19$ and the following result (a corollary of \cite[Lemma~3.3]{FKPT13}) 
shows just how deceiving
the situation is in small dimensions. 

\begin{prop}\label{prop:rcbounds}
    The rectangle covering number of the $(n,k)$-hypersimplex satisfies
    \[
        n \ \leq \ \rc(\Delta_{n,k}) \ \leq \  n+\lceil\mathsf{e}\,(k+1)^2\log(n)\rceil.
    \]
\end{prop}

\begin{proof}
    The lower bound follows from Corollary~\ref{cor:lb_rc}.  For the upper
    bound, consider the matrix {$\Gnk{n}{k}$} whose columns are the $0/1$
    vectors with $k$ zeros.  If $2\leq k\leq n-2$, the rows of the slack
    matrix of $\Delta_{n,k}$ induced by the $G_i$ facets provide a copy of
    $\Gnk{n}{k}$, and the rows induced by $F_i$ facets a copy of
    $\Gnk{n}{n-k}$.  Thus, 
    \[
        \rc(\Delta_{n,k}) \ \leq \  \rc(\Gnk{n}{n-k})+\rc(\Gnk{n}{k}).
    \]
    Observing that $\rc(\Gnk{n}{n-k})$ is trivially bounded from above by $n$
    (take a rectangle for each row), it suffices to see that
    $\rc(\Gnk{n}{k})\leq \lceil\mathsf{e}\,(k+1)^2\log(n)\rceil$, which is
    shown in~\cite[Lemma~3.3]{FKPT13}.
    
    We reproduce their nice argument for completeness. The rows and columns of
    $\Gnk{n}{k}$ are indexed by the sets $[n]$ and $\binom{[n]}{n-k}$,
    respectively.  The non-zero elements are the pairs $(x,S) \in [n] \times
    \binom{[n]}{n-k}$ with $x\in S$. The inclusion-maximal rectangles are of
    the form 
    \[
        R_I \ := \ \left\{(x,S)\in[n]\times \binom{[n]}{n-k}\ : \ x\in I
        \text{ and } I \subseteq S\right\},
    \] 
    for $I\subseteq [n]$. We can pick an $I$ at random by selecting every
    element in $I$ independently with probability $p=\frac{1}{k+1}$.  The
    probability then that an entry $(x,S)$ with $x\in S$ is covered by $R_I$
    is $p(1-p)^k$. Hence, if we choose $r =
    \lceil\mathsf{e}\,(k+1)^2\log(n)\rceil$ such rectangles $R_I$
    independently, then probability that an entry is not covered by any of the
    rectangles is $(1-p(1-p)^k)^r$.

    The total number of non-zero entries of $\Gnk{n}{k}$ is $(n-k)\cdot
    \binom{n}{k}<n^{k+1}$. Therefore, the 
    logarithm of the expected number of non-zero entries of $\Gnk{n}{k}$ that are not covered by any rectangle is at most
    \begin{align*}\log\left((1-p(1-p)^k)^rn^{k+1}\right)&=r\log\left(1-p(1-p)^k\right)+(k+1)\log(n)\\
							&\leq - r p(1-p)^k +(k+1)\log(n)=- r \frac{k^k}{(k+1)^{k+1}} +(k+1)\log(n).
    \end{align*}
    If this is negative, then there is at least one covering with $r$ rectangles. That is, whenever
    \begin{align*} r  &>\frac{(k+1)^{k+2}}{k^k}\log(n).
    \end{align*}
    Observing that
    \begin{align*}
    \frac{(k+1)^{k+2}}{k^k}\log(n) \ = \
    (k+1)^2\left(\frac{k+1}{k}\right)^k\log(n) \ < \ 
    \mathsf{e}\,(k+1)^2\log(n)
    \end{align*}
    concludes the proof.    
\end{proof}

\section{Combinatorial hypersimplices and realization spaces}
\label{sec:relspaces}

Typically, the extension complexity is not an invariant of the combinatorial
isomorphism class of a polytope (see, for example, the situation with
hexagons~\cite{polygons2}). However, Corollary~\ref{cor:cube} states
that every combinatorial cube, independent of its realization, has the same
extension complexity. The proximity to cubes and the results in
Sections~\ref{sec:geom} and~\ref{sec:rc} raised the hope that this extends
to all hypersimplices. A \Defn{combinatorial $(n,k)$-hypersimplex} is any
polytope whose face lattice is isomorphic to that of $\Delta_{n,k}$. One
approach would have been through rectangle covering numbers but
Proposition~\ref{prop:rcbounds} refutes this approach in the strongest possible
sense.

We extend the notions of $F$- and $G$-facets from~\eqref{eqn:FG} to
combinatorial hypersimplices. The crucial property that we used in the proof
of Proposition~\ref{prop:large_cases} was that in the standard realization of
$\Delta_{n,k}$, the polyhedron bounded by hyperplanes supporting the
$G$-facets is a full-dimensional simplex. We call a combinatorial hypersimplex
\Defn{$\mathbf{G}$-generic} if the hyperplanes supporting the $G$-facets are
not projectively concurrent, that is, if the hyperplanes supporting
combinatorial $(n-1,k-1)$-hypersimplices do not meet in a point and are not
parallel to a common line.  We define the notion of
\Defn{$\mathbf{F}$-generic} hypersimplices likewise and we simply write
\Defn{$\mathbf{FG}$-generic} if a hypersimplex is $F$- and $G$-generic.

Now, if a combinatorial hypersimplex $P$ is $G$-generic, then there is an
admissible projective transformation that makes the polyhedron induced by the
$G$-facets bounded. To find such a transformation, one can proceed as follows:
translate $P$ so that it contains $0$ in the interior, then take the polar
$P^\circ$ and translate it so that the origin belongs to the interior of the
convex hull of the $G$-vertices. This is possible because $G$-genericity
implies that these vertices span a full-dimensional simplex. Taking the polar
again yields a polytope $P'$ that is projectively equivalent to $P$.  Since
projective transformations leave the extension complexity invariant, the proof
of Proposition~\ref{prop:large_cases}, almost verbatim, carries over to
$FG$-generic hypersimplices.

Indeed, with the upcoming Lemma~\ref{lem:generic}, it is straightforward to
verify that $F$-facets of an $FG$-generic $(n,k)$-hypersimplex with $2k\geq n$
are again $FG$-generic; and the same works with $G$-facets when $2k\leq n$.
Hence, one can apply the inductive reasoning of the $k=2$ case of
Proposition~\ref{prop:large_cases} and together with
Proposition~\ref{prop:small_cases}, this proves the following theorem.

\begin{thm}\label{thm:main_generic}
    If $P$ is an $FG$-generic combinatorial $(n,k)$-hypersimplex with $n\geq
    6$ and $2\leq k\leq n-2$, then $\rkN(P)=2n$.
\end{thm}

The following lemma states that \emph{every} combinatorial hypersimplex is
either $F$-generic or $G$-generic.

\begin{lem}\label{lem:generic}
    Every combinatorial $(n,k)$-hypersimplex is
    \begin{enumerate}[\rm (i)]
        \item $F$-generic if $2k<n+2$, and 
        \item $G$-generic if $2k>n-2$.
    \end{enumerate}
    In particular, every combinatorial $(n,k)$-hypersimplex is $FG$-generic
    for $n-2<2k<n+2$. 
\end{lem}
\begin{proof}
    The two statements are dual under the affine equivalence $\Delta_{n,k} \cong
    \Delta_{n,n-k}$. Hence, we only prove the second statement. For this, let
    $P^\circ$ be polar to a combinatorial $(n,k)$-hypersimplex with $2k >
    n-2$. Thus, $P^\circ$ is a polytope of dimension $n-1$ with $2n$ vertices
    $f_1,\dots,f_n$ and $g_1,\dots,g_n$ corresponding to the $F$- and
    $G$-facets. In this setting, $G$-genericity means that the polytope $Q =
    \conv(g_1,\dots,g_n)$ is of full dimension $n-1$.
    From the combinatorics of $(n,k)$-hypersimplices, we infer that for every
    $I \subseteq [n]$ with $|I| = k$, the set 
    \[
        \conv \left( \{g_i : i \in I\} \cup \{f_i : i \not\in I \} \right)
    \]
    is a face of $P^\circ$ and hence $\conv(g_i : i \in I)$ is a face of $Q$.
    Thus $Q$ is a \emph{$k$-neighborly} polytope with $2k \ge n-1 \ge \dim Q$.
    It follows from~\cite[Thm.~7.1.4]{Grunbaum} that $Q$ is a simplex and
    thus of dimension $n-1$.
\end{proof}

Although we will later see that not every combinatorial hypersimplex is
$FG$-generic (cf. Proposition~\ref{prop:singular62}), this has some
immediate consequences for the extension complexity of combinatorial
hypersimplices.  The following corollary can be deduced from
Figure~\ref{fig:hypersimplex_chart}, using Corollary~\ref{cor:face} to
navigate along the arrows to the (thick) diagonal.

\begin{repcor}{cor:XCcombinatorial}
    If $P$ is a combinatorial $(n,k)$-hypersimplex with $n\geq 6$, then
    \[
        \renewcommand\arraystretch{1.2}
        \rkN(P) \ \geq \  
        \left\{
        \begin{array}{rrr@{}c@{}l}
            n+2k+1,& \text{ if} & 2 \ \le \ & k & \ < \ \ffloor{n}{2},\\
            2n,& \text{ if}  & \ffloor{n}{2} \ \le \ &k & \ \le \
                                        \fceil{n}{2}, \text{ and}\\
            n+2(n-k)+1,& \text{ if} & \fceil{n}{2} \ < \ &k & \ \le \  n-2.
        \end{array}
        \right.
    \]
\end{repcor}
\begin{proof}
    For $\ffloor{n}{2} \le k \le \fceil{n}{2}$, we get the result as a
    combination of Theorem~\ref{thm:main_generic} with
    Lemma~\ref{lem:generic}.  If $P$ is a combinatorial $(n,k)$-hypersimplex
    with $k<\ffloor{n}{2}$, then $P$ has a $2k$-dimensional face $Q$
    isomorphic to $\Delta_{2k+1,k}$ (by successively taking $F$-facets). By
    the previous case, $\rkN(Q)=2(2k+1)$. By Corollary~\ref{cor:face},
    $\rkN(P)\geq n+2k+1$. The case $k>\fceil{n}{2}$ follows symmetrically.
\end{proof}

\begin{figure}[htpb]
 
\makeatletter

\pgfdeclareshape{rectangle with diagonal fill}
{
    \inheritsavedanchors[from=rectangle]
    \inheritanchorborder[from=rectangle]
    \inheritanchor[from=rectangle]{north}
    \inheritanchor[from=rectangle]{north west}
    \inheritanchor[from=rectangle]{north east}
    \inheritanchor[from=rectangle]{center}
    \inheritanchor[from=rectangle]{west}
    \inheritanchor[from=rectangle]{east}
    \inheritanchor[from=rectangle]{mid}
    \inheritanchor[from=rectangle]{mid west}
    \inheritanchor[from=rectangle]{mid east}
    \inheritanchor[from=rectangle]{base}
    \inheritanchor[from=rectangle]{base west}
    \inheritanchor[from=rectangle]{base east}
    \inheritanchor[from=rectangle]{south}
    \inheritanchor[from=rectangle]{south west}
    \inheritanchor[from=rectangle]{south east}

    \inheritbackgroundpath[from=rectangle]
    \inheritbeforebackgroundpath[from=rectangle]
    \inheritbehindforegroundpath[from=rectangle]
    \inheritforegroundpath[from=rectangle]
    \inheritbeforeforegroundpath[from=rectangle]

    \behindbackgroundpath{%
        \pgfextractx{\pgf@xa}{\southwest}%
        \pgfextracty{\pgf@ya}{\southwest}%
        \pgfextractx{\pgf@xb}{\northeast}%
        \pgfextracty{\pgf@yb}{\northeast}%
        \ifpgf@diagonal@lefttoright
            \def\pgf@diagonal@point@a{\pgfpoint{\pgf@xa}{\pgf@yb}}%
            \def\pgf@diagonal@point@b{\pgfpoint{\pgf@xb}{\pgf@ya}}%
        \else
            \def\pgf@diagonal@point@a{\southwest}%
            \def\pgf@diagonal@point@b{\northeast}%
        \fi
        \pgfpathmoveto{\pgf@diagonal@point@a}%
        \pgfpathlineto{\northeast}%
        \pgfpathlineto{\pgfpoint{\pgf@xb}{\pgf@ya}}%
        \pgfpathclose
        \ifpgf@diagonal@lefttoright
            \color{\pgf@diagonal@top@color}%
        \else
            \color{\pgf@diagonal@bottom@color}%
        \fi
        \pgfusepath{fill}%
        \pgfpathmoveto{\pgfpoint{\pgf@xa}{\pgf@yb}}%
        \pgfpathlineto{\southwest}%
        \pgfpathlineto{\pgf@diagonal@point@b}%
        \pgfpathclose
        \ifpgf@diagonal@lefttoright
            \color{\pgf@diagonal@bottom@color}%
        \else
            \color{\pgf@diagonal@top@color}%
        \fi
        \pgfusepath{fill}%
    }
}

\newif\ifpgf@diagonal@lefttoright
\def\pgf@diagonal@top@color{white}
\def\pgf@diagonal@bottom@color{gray!30}

\def\pgfsetdiagonaltopcolor#1{\def\pgf@diagonal@top@color{#1}}%
\def\pgfsetdiagonalbottomcolor#1{\def\pgf@diagonal@bottom@color{#1}}%
\def\pgfsetdiagonallefttoright{\pgf@diagonal@lefttorighttrue}%
\def\pgfsetdiagonalrighttoleft{\pgf@diagonal@lefttorightfalse}%

\tikzoption{diagonal top color}{\pgfsetdiagonaltopcolor{#1}}
\tikzoption{diagonal bottom color}{\pgfsetdiagonalbottomcolor{#1}}
\tikzoption{diagonal from left to right}[]{\pgfsetdiagonallefttoright}
\tikzoption{diagonal from right to left}[]{\pgfsetdiagonalrighttoleft}
\makeatother

\begin{tikzpicture}[scale=1.1]
\foreach \n in {3,...,9} {
  \foreach \j [evaluate=\j as \k using int(\j-1)] in {2,...,\n} {
  \pgfmathparse{int(2*\k-\n)}
  \ifnum\pgfmathresult>1
	\node[rectangle, draw, very thick, fill=red!25, draw=red] (\n\k) at (2*\n,\n/2-\k) {$\Delta_{\n,\k}$};
  \else
	  \ifnum\pgfmathresult<-1
		 \node[rectangle, draw, very thick, fill=blue!25, draw=blue] (\n\k) at (2*\n,\n/2-\k) {$\Delta_{\n,\k}$};
 	 \else
    		\node[rectangle with diagonal fill, draw, very thick, diagonal top color=blue!25, diagonal bottom color=red!25,draw=black] (\n\k) at (2*\n,\n/2-\k) {$\Delta_{\n,\k}$};
  	\fi
  \fi
  }
}

  \foreach \n [evaluate=\n as \m using int(\n+1)] in {3,...,8} {
  \foreach \j [evaluate=\j as \k using int(\j-1)] in {2,...,\n} {
     \draw [latex-,red, very thick] (\n\k) -- node[midway,fill=white] {$G$} (\m\j);
     \draw [latex-,blue, very thick] (\n\k) -- node[midway,fill=white] {$F$}(\m\k);
  }
}

\end{tikzpicture}
\caption{The smaller hypersimplices. Arrows represent the structure of the $F$- and $G$-facets. Those in the upper half are $F$-generic, those in the lower half are $G$-generic and those in the middle are $FG$-generic.}\label{fig:hypersimplex_chart}
\end{figure}
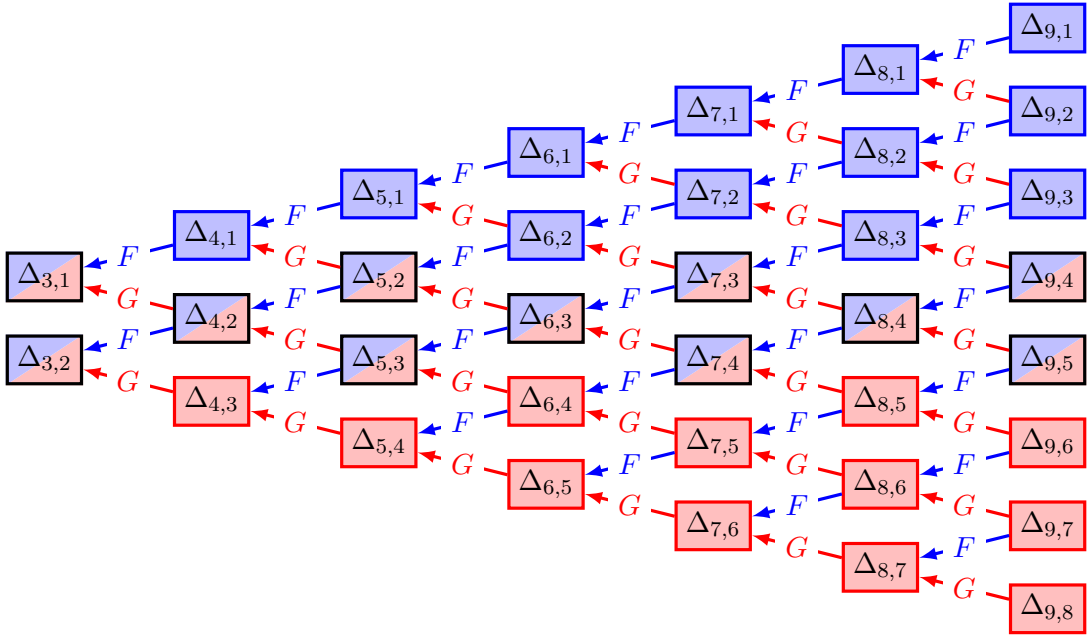

\subsection{Realization spaces of hypersimplices}
A combinatorial $(n,k)$-hypersimplex is a polytope $P \subset \R^{n-1}$ given
by $2n$ linear inequalities $f_i(\x) = f_{i0} + \sum_j f_{ij}x_j \ge 0$ and
$g_i(\x) = g_{i0} + \sum_j g_{ij}x_j \ge 0$ for $i=1,\dots,n$ such that $P$ is
combinatorially isomorphic to $\Delta_{n,k}$ under the correspondence
\begin{align*}
    F_i = \{ \x \in \Delta_{n,k} : x_i = 0\} &\quad \longrightarrow \quad \{ \x
    \in P : f_i(\x) = 0\}, \text{ and} \\
    G_i = \{ \x \in \Delta_{n,k} : x_i = 1\} &\quad \longrightarrow \quad \{ \x
    \in P : g_i(\x) = 0\}.
\end{align*}
Of course, the inequalities are unique only up to a
positive scalar and hence the group $(\R^{2n}_{>0},\cdot)$ acts on ordered
collections of linear inequalities furnished by all combinatorial
$(n,k)$-hypersimplices in $\R^{n-1}$. We only want to consider realizations that
are genuinely distinct and it is customary to identify two realizations of
$\Delta_{n,k}$ that differ by an affine transformation or an (admissible) projective transformation; see,
for example, \cite[Sect.  2.1]{RichterGebert1997}
or~\cite[Sect.~8.1]{OrientedMatroids1993}. We do the latter. However, care has to be taken as
the projective linear group does not act on the realization space. To that
end, we identify $P$ with its \Defn{homogenization} 
\[
    \homog(P) \ := \ \cone( \{1 \} \times P)  \ = \
    \left\{ \binom{x_0}{\x} \in \R^n : 
    \begin{array}{r@{ \ \ge \ }l}
    g_{0i}x_0 + \dots + g_{ni}x_n & 0 \\
    f_{0i}x_0 + \dots + f_{ni}x_n & 0 \\
    \end{array}
    \text{ for } i=1,\dots,n
    \right\}.
\]
Under this identification, one verifies that two $(n,k)$-hypersimplices $P$ and
$P'$ are projectively equivalent if and only if $\homog(P)$ and $\homog(P')$
are linearly isomorphic. The \Defn{projective realization space} $\Rel_{n,k}$
of combinatorial $(n,k)$-hypersimplices is the set of matrices
$(g_1,\dots,g_n,f_1,\dots,f_n) \in \R^{n \times 2n}$ that yield cones
isomorphic to the homogenization of the $(n,k)$-hypersimplex modulo the action
of $\mathrm{GL}_n \times \R^{2n}_{>0}$. We write $g_i^\perp$ and $f_i^\perp$
for the facet defining hyperplanes of $\homog(P)$ corresponding to $g_i$ and
$f_i$, respectively.

Let us fix $1 < k \le \frac{n}{2}$ and let $P$ be a combinatorial
$(n,k)$-hypersimplex. By Lemma~\ref{lem:generic}, the $F$-facets are generic
and hence bound a simplex $Q$ up to projective transformation. That is,
\[
    \homog(Q) \ = \ \{ \x \in \R^n : 
    f_{0i}x_0 + \dots + f_{ni}x_n \ge 0 \text{ for } i=1,\dots,n \} \ \cong \
    \R^n_{\ge 0}
\]
by a linear transformation. Hence, we can choose a matrix representing
$\homog(P)$ of the form
\begin{equation}\label{eqn:G_proj}
    \left(
   \begin{array}{cccc|ccc}
                | &| & &| & 1 & & \\
                g_1&g_2&\cdots&g_n & & \ddotsHacked &\\
                | &| & & | & & &1\\
    \end{array}
    \right).
\end{equation}
Modulo positive column and row scaling, the matrix $(g_1,\dots,g_n)$ uniquely determines
$P$ up to projective transformations. 
Indeed, by using suitable positive column scaling on the $f_i$'s, 
the effect of positively scaling rows of~\eqref{eqn:G_proj} leaves
the identity matrix of~\eqref{eqn:G_proj} invariant. 

For example, a representative of the standard realization of $\Delta_{n,k}$ is given by the $G$-matrix:
\begin{equation}\label{eqn:standardGmatrix}
\left(
       \begin{array}{c c c c}
	-k+1 & 1 & \cdots  & 1\\
	1 & -k+1 & \cdots  & 1\\
	\vdots & \vdots& \ddots  & \vdots   \\ 
	1 & 1& \cdots  & -k+1	
       \end{array}\right).
\end{equation}

The $G$-matrix also gives us a condition for $G$-genericity. The hyperplanes
to $G$-facets are projectively concurrent if there is a nonzero element in
the kernel of the $G$-matrix. This proves the following useful criterion.

\begin{lem}\label{lem:G_generic}
    Let $P$ be a combinatorial $(n,k)$-hypersimplex with $k \le \frac{n}{2}$.
    Then $P$ is $G$-generic if and only if $\det(g_1,\dots,g_n) \neq 0$.
\end{lem}

Notice that the principal $k$-minors of~\eqref{eqn:standardGmatrix} vanish. 
This is common to all
combinatorial hypersimplices.  Indeed, the combinatorics of hypersimplices
dictates that for any $I \subseteq [n]$ with $|I| = k$, the intersection of
the hyperplanes $\{g_i^\perp : i \in I\}$ and $\{f_i^\perp : i \not\in I\}$
with $\homog(P)$ is a face of dimension $1$.  By our choice of
$f_1,\dots,f_n$, this is equivalent to
\begin{equation}\label{eqn:G_pm}
    \rk(G_{II}) \ = \ k-1 \text{ for all } I \subseteq [n] \text{ with } |I| =
    k.
\end{equation}

Notice that these are the only equality constrains for $\Rel_{n,k}$.
The remaining conditions are only strict inequalities coming vertex-facet-nonincidences.

The \emph{complex} variety of $n$-by-$n$ matrices with vanishing principal
$k$-minors was studied by Wheeler in~\cite{Wheeler15}, which turns out to be a
rather complicated object. For example, it is not known if the variety is
irreducible and even the dimension is only known in certain cases. The
following result yields an upper bound on the dimension of $\Rel_{n,k}$.

\begin{thm}\label{thm:Rel_UB}
     For $2<k< n-2$, every $P \in \Rel_{n,k}$ is completely identified by
     a realization of $\Delta_{n-1,k-1}$ unique up to affine transformation.
    In particular, for $2\leq k\leq n-2$, the dimension of $\Rel_{n,k}$ is at
    most $\binom{n-1}{2}$.
\end{thm}
\begin{proof}
    Let $P$ be a combinatorial $(n,k)$-hypersimplex. By a suitable projective
    transformation, we can assume that the facet-defining hyperplanes of $F_1$
    and $G_1$ are parallel. This assumption fixes the intersection of
    $\aff(G_1)$ with the hyperplane at infinity and hence fixes $G_1$ up to an
    affine transformation. 
    
    Since $F_1$ and $G_1$ lie in parallel hyperplanes, the corresponding
    facets of $F_1$ and $G_1$ are parallel (because they are induced by the
    intersection of the same supporting hyperplanes of $\Delta_{n,k}$ with
    these two parallel hyperplanes).  
    
    A result of Shephard (see~\cite[Thm.~15.1.3, p.321]{Grunbaum}) states that
    if all the $2$-dimensional faces of a polytope $R \subset \R^d$ are
    triangles, then for any representation $R = R_1 + R_2$ of $R$ as
    Minkowski sum, there are $t_i \in
    \R^d$ and $\lambda_i \ge 0$ such that  $R_i = t_i + \lambda_i R$ for
    $i=1,2$. Now, if $Q$ and $Q'$ are \emph{normally equivalent} polytopes, i.e.\
    combinatorially equivalent and corresponding facets are parallel, and $Q$
    has only triangular $2$-faces, then, by~\cite[Prop.~7.12]{Ziegler1995},
    all $2$-faces of $Q+Q'$ are triangles as well. It follows that $Q$ and
    $Q'$ are positively homothetic.

    Since every face of a hypersimplex is a hypersimplex and $2$-dimensional
    hypersimplices are triangles, this shows that realizations of
    hypersimplices are determined up to positive homothety once their facet
    directions are determined.  In particular, this shows that given $G_1$,
    $F_1$ is determined up to a positive homothety. Hence, given $G_1$, $P$ is
    unique up to projective transformations.
    
    The bound on the dimension follows by induction on $k$. We will see in
    Theorem~\ref{thm:dim_rel_space} that $\dim\Rel_{n,2}=\binom{n-1}{2}$,
    settling the base case.  The affine group of $\R^d$ is a codimension~$d$
    subgroup of the projective group. Hence, by induction, 
    \[
        \dim\Rel_{n,k} \ \leq \  \dim\Rel_{n-1,k-1}+(n-2) \ \leq \ 
        \binom{n-2}{2}+(n-2) \ = \ \binom{n-1}{2}.\qedhere
    \] 
\end{proof}

\section{The $(n,2)$-hypersimplices}\label{sec:n2}

Although realization spaces are notoriously difficult objects and it is
generally difficult to access different realizations of a given polytope, in
the case of $(n,2)$-hypersimplices we have a simple construction and a nice
geometrical interpretation. Let us denote by $\Delta_{n-1} =
\conv(e_1,\dots,e_n) \subset \R^n$ the \Defn{standard simplex} of dimension
$n-1$.

\begin{thm}\label{thm:simplex_edges}
    For $n \ge 4$, let $p_{ij}$ be a point in the relative interior of the
    edge $[e_i,e_j] \subset \Delta_{n-1}$ for $1 \le i < j \le n$. Then
    \[
        P \ := \ \conv\{ p_{ij} : 1 \le i < j \le n \}
    \]
    is a combinatorial $(n,2)$-hypersimplex. Up to projective transformation,
    every combinatorial $(n,2)$-hypersimplex arises this way.
\end{thm}
\begin{proof}
    Since $\Delta_{n-1}$ is a simple polytope, the polytope $P$ is the result
    of truncating every vertex $e_i$ of $\Delta_{n-1}$ by the unique
    hyperplane spanned by $\{p_{ij} : j \neq i\}$. Hence, $P$ has
    $\binom{n}{2}$ vertices and every $p_{ij}$ is incident to exactly two
    facets isomorphic to $\Delta_{n-1,1} \cong \Delta_{n-2}$. If $n=4$, then
    it is easily seen that $P$ is an octahedron. For $n > 4$, we get by
    induction on $n$ that the remaining facets $P \cap \{ x : x_i = 0 \}$ are
    isomorphic to $\Delta_{n-1,2}$, which implies the first claim.

    For the second statement, let $P$ be a combinatorial $(n,2)$-hypersimplex.
    We know from Lemma~\ref{lem:generic} that the
    $F$-facets bound a projective simplex. By a
    suitable projective transformation, we may assume that this is exactly
    $\Delta_{n-1}$.  Each vertex of $P$ lies in all the $F$-facets except for
    two. So every vertex of $P$ lies in the relative interior of a unique edge
    of $\Delta_{n-1}$.
\end{proof}

Note that the representation given in Theorem~\ref{thm:simplex_edges} is not
unique up to projective transformation.  The simplest way to factor out the
projective transformations is to perform a first truncation at the vertex
$e_1$ of $\Delta_{n-1}$. Then $\conv\{p_{12},\dots, p_{1n}, e_2,\dots, e_n\}$
is a prism over a simplex, which is projectively unique
(cf.~\cite[Ex.~4.8.30]{Grunbaum}). It then only remains to choose
$\binom{n-1}{2}$ points in the interior of every edge of the base of the
prism. Each choice produces a projectively distinct $(n,2)$-hypersimplex, and
every $(n,2)$-hypersimplex arises this way, up to projective transformation.
This completes the proof of Theorem~\ref{thm:dim_rel_space}.

\begin{repthm}{thm:dim_rel_space}
    For $n \geq 4$, $\Rel_{n,2}$ is {rationally} equivalent to the
    interior of a $\binom{n-1}{2}$-dimensional cube. In particular,
    $\Rel_{n,2}$ is homeomorphic to an open ball and hence contractible. 
\end{repthm}

To recover the description as an $n\times n$ matrix, we can proceed as
follows. Set the (projective) simplex $\Delta_F$ bounded by the $F$-facets to
be the standard simplex. Now, for every (oriented) edge $[e_i,e_j]$ of
$\Delta_F$, consider the ratio
$\rho_{ij}=\frac{\|e_i-p_{ij}\|}{\|e_j-p_{ij}\|}$ for $i \neq j$.  It is not
hard to see that the diagonal entries of the $G$-matrix are negative and that,
if we scale its columns so that they are $-1$, then we are left with the
matrix
\begin{equation}\label{eq:MatDeltan2}
\left(
       \begin{array}{c c c c}
	-1 & \rho_{12} & \cdots  & \rho_{1n}\\
	\rho_{21} & -1 & \cdots  & \rho_{2n}\\
	\vdots & \vdots& \ddots  & \vdots   \\ 
	\rho_{n1} & \rho_{n2}& \cdots  & -1	
       \end{array}\right);
\end{equation}
containing $-1$ in the diagonal and the ratios $\rho_{ij}$ in the remaining entries. Notice how the condition on the vanishing $2\times 2$ principal minors coincides with the relation $\rho_{ij}=\rho_{ji}^{-1}$. By Theorem~\ref{thm:simplex_edges}, any choice of positive ratios fulfilling this relation gives rise to a realization of an $(n,2)$-hypersimplex.

All non-diagonal entries are positive by construction. Multiplying the $i$th
column by $\rho_{i1}$ and the $i$th row by $\rho_{1i}$ for $2\leq i\leq n$,
which corresponds to a projective transformation, we are left with an
equivalent realization. Relabelling the ratios, it is of the form:
\begin{equation}\label{eq:ProjDeltan2}
\left(
       \begin{array}{c c c c}
	-1 & \phantom{-}1 & \cdots  & \phantom{-}1\\
	\phantom{-}1 & -1 & \cdots  & \rho_{2n}\\
	\phantom{-}\vdots & \phantom{-}\vdots& \ddots  & \phantom{-}\vdots   \\ 
	\phantom{-}1 & \rho_{n2}& \cdots  & -1	
       \end{array}\right).
\end{equation}
Choosing a realization of $\Delta_{n,2}$ up to projective transformation amounts to choosing $\binom{n-1}{2}$ positive ratios and we recover the description of Theorem~\ref{thm:dim_rel_space}.
Actually, this is the transformation that we also used before, which transforms the truncated simplex into a prism over a simplex. The $\binom{n-1}{2}$ remaining ratios correspond to the edges of the basis of the prism.

Notice that, although a similar argumentation provides the realization space up to affine transformation, that setup is slightly more delicate. While a choice of a point on each of the 
$\binom{n}{2}$ edges of a standard simplex gives a unique affine realization of $\Delta_{n,2}$, not every affine realization can be obtained this way: The $F$-simplex might be unbounded for some realizations of $\Delta_{n,2}$. Hence, there are several ``patches'' in the affine realization space, each one corresponding to a different relative position of the $F$-simplex with respect to the hyperplane at infinity. For every patch, realizations are parametrized by the position of the points in the edges of $\Delta_F$.

\begin{ex}
 To show an example, we will work with $\Delta_{3,2}$, which we look at as two nested projective simplices, the second being the convex hull of a point in each edge of the first. Even though this is not strictly a hypersimplex, it provides simpler figures than $\Delta_{4,2}$. Figure~\ref{fig:Deltan2} depicts four such realizations.
  \begin{figure}[htpb]
  \includegraphics[width=\linewidth]{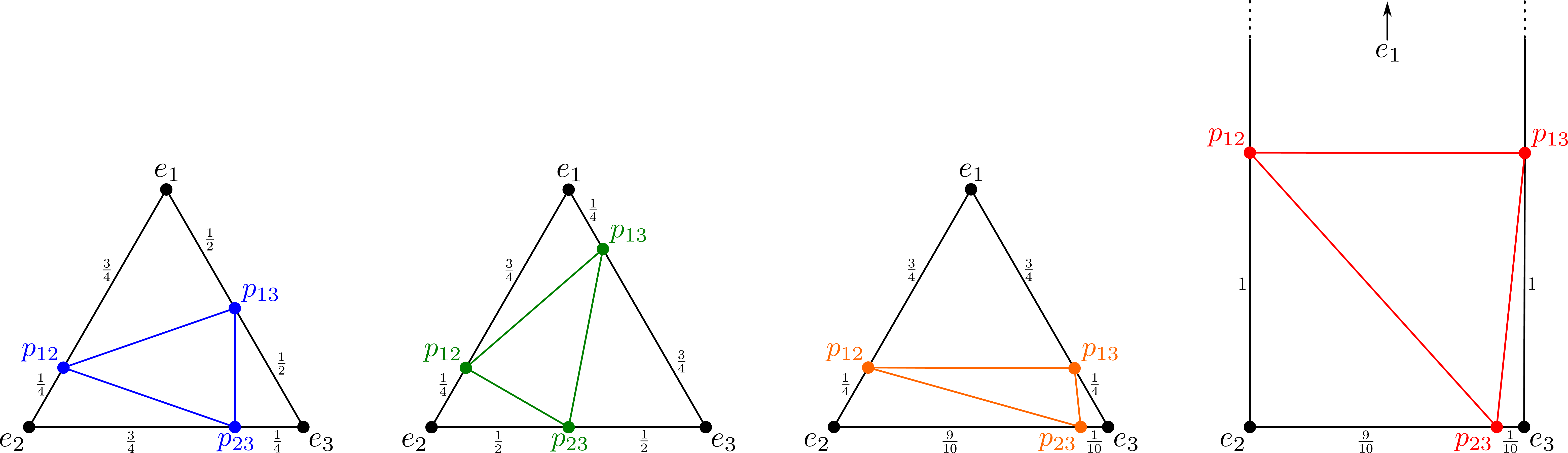}
  \caption{Four (projectively equivalent) realizations of $\Delta_{3,2}$ as two nested (projective) simplices.}\label{fig:Deltan2}
 \end{figure}

 In the first three, the outer simplex is a standard simplex. Computing the ratios of~\eqref{eq:MatDeltan2} we recover the matrices
 \begin{align*}
  \left(
       \begin{array}{c c c}
	-1 & \phantom{-}3 & \phantom{-}1\\
	\phantom{-}\frac13 & -1 & \phantom{-}3\\
	\phantom{-}1 & \phantom{-}\frac13 & -1	
       \end{array}\right)\text{, } \qquad\qquad
  \left(
       \begin{array}{c c c}
	-1 & \phantom{-}3 & \phantom{-}\frac13\\
	\phantom{-}\frac13 & -1 & \phantom{-}1\\
	\phantom{-}3 & \phantom{-}1 & -1	
       \end{array}\right)
       \qquad \text{ and } \qquad
  \left(
       \begin{array}{c c c}
	-1 & \phantom{-}3 & \phantom{-}3\\
	\phantom{-}\frac13 & -1 & \phantom{-}9\\
	\phantom{-}\frac13 & \phantom{-}\frac19 & -1	
       \end{array}\right).
 \end{align*}
We can transform the first into the second (resp.\ third) by multiplying the
third row by $3$ (resp.\ $\frac13$) and the third column by $\frac13$ (resp.\ 
$3$), which means that they represent projectively equivalent realizations. To
fix a unique projective representative, we send $e_1$ to infinity, and impose
that $p_{12},p_{13},e_2,e_3$ form a prism over a standard simplex (in this
case a square), as in the fourth figure. The ratios that we get on the base of
the prism give the entries of the matrix corresponding to
\eqref{eq:ProjDeltan2}:
\[
  \left(
       \begin{array}{c c c}
	-1 & \phantom{-}1 & \phantom{-}1\\
	\phantom{-}1 & -1 & \phantom{-}9\\
	\phantom{-}1 & \phantom{-}\frac19 & -1	
       \end{array}\right).
\]
In this example, this projective transformation corresponds to multiplying the second row and the third column of the first matrix by $3$, and its third row and second column by $\frac13$.
\end{ex}

With the aid of this description, we can produce our first example of a non
$FG$-generic hypersimplex. The following is due to Francisco Santos (personal
communication), who found nicer coordinates than our original example.

\begin{prop}\label{prop:singular62}
    There are $(6,2)$-hypersimplices that are not $G$-generic.
\end{prop}
\begin{proof}
    The following matrix corresponds to a non-$G$-generic realization of
    $\Delta_{6,2}$
    \[
\left(
       \begin{array}{c c c c c c}
	-1 & \phantom{-}2\sqrt{6}+5 & -2 \sqrt{6} +5 & \phantom{-}1 & \phantom{-}1 & \phantom{-}1\\
	-2\sqrt{6}+5 & -1 & \phantom{-}2\sqrt{6}+5 & \phantom{-}1& \phantom{-}1& \phantom{-}1\\
	\phantom{-}2\sqrt{6}+5 & -2\sqrt{6}+5 & -1 & \phantom{-}1& \phantom{-}1 & \phantom{-}1 \\
	\phantom{-}1& \phantom{-}1& \phantom{-}1& -1& \phantom{-}1& \phantom{-}1\\
	\phantom{-}1& \phantom{-}1& \phantom{-}1& \phantom{-}1& -1& \phantom{-}1\\
	\phantom{-}1& \phantom{-}1& \phantom{-}1& \phantom{-}1& \phantom{-}1& -1
       \end{array}\right).
\]
    It can be checked that the determinant is zero and
    Lemma~\ref{lem:G_generic} finishes the proof.
\end{proof}

\begin{cor}\label{cor:prescribability}
Not every combinatorial $(n,k)$-hypersimplex is a facet of a combinatorial $(n+1,k+1)$-hypersimplex.
\end{cor}
\begin{proof}
Every combinatorial $(7,3)$-hypersimplex is $FG$-generic by
Lemma~\ref{lem:generic}, a property that is inherited by its $G$-facets.
Hence, the combinatorial $(6,2)$-hypersimplex of Proposition~\ref{prop:singular62} is not a facet of any combinatorial $(7,3)$-hypersimplex.
\end{proof}

In contrast, the shape of facets of $(n,2)$-hypersimplices can be prescribed. This is a direct corollary of 
the construction from Theorem~\ref{thm:simplex_edges}.

\begin{cor}\label{cor:n2prescribability}
Every combinatorial $(n,2)$-hypersimplex is a facet of a combinatorial $(n+1,2)$-hypersimplex.
\end{cor}

\subsection{The $(5,2)$-hypersimplex}\label{sec:52}

The $(5,2)$-hypersimplex is special. As we argued in the proof of
Proposition~\ref{prop:small_cases}, the nonnegative rank of the hypersimplex
$\Delta_{5,2}$ in its defining realization~\eqref{eqn:hyper} is equal to $9$.
In light of Theorem~\ref{thm:main}, this deviates from the `expected'
nonnegative rank.  The goal of this section is to show that \emph{generic}
realizations of $\Delta_{5,2}$ have nonnegative rank $10$.

We just described the (projective) realization space of $\Delta_{5,2}$  obtained by
choosing an interior point in every edge of the base of $\Delta_3\times \Delta_1$. That means that
$\dim\Rel_{5,2}=6$.

We now claim that the realization $P$ of $\Delta_{5,2}$  given as the
convex hull of the columns
\begin{equation}\label{eqn:52special}
    \left(\begin{array}{rrrrrrrrrr}
        35 & 35 & 35 & 35 & 0 & 0 & 0 & 0 & 0 & 0 \\
        35 & 0 & 0 & 0 & 50 & 42 & 20 & 0 & 0 & 0 \\
        0 & 35 & 0 & 0 & 20 & 0 & 0 & 56 & 60 & 0 \\
        0 & 0 & 35 & 0 & 0 & 28 & 0 & 14 & 0 & 42 \\
        0 & 0 & 0 & 35 & 0 & 0 & 50 & 0 & 10 & 28
    \end{array}\right)
\end{equation}
has nonnegative rank $10$. To prove this, we again use a computer to compute
the refined rectangle covering number from~\cite{OVW14}. Like the ordinary
rectangle covering number, the refined rectangle covering number yields lower
bounds on the nonnegative rank of a polytope $P$ but instead of only the
support pattern of the slack matrix $S_P$, it takes into account simple
relations between the actual values. A covering of $S = S_P$ by rectangles
$R_1,\dots,R_s$ is \Defn{refined} if for any pair of indices $(i,k),(j,l)$
such that
\begin{equation}\label{eqn:rrc}
    S_{ik} S_{jl} \ > \ S_{il}S_{jk},
\end{equation}
then the number of rectangles containing $S_{ik}$ or $S_{jl}$ is at least two.
(necessarily positive) entries $S_{ik},S_{jl}$ are contained in at least
two rectangles. Of course, if $S_{il}$ or $S_{jk}$ are zero, then this reduces
to the condition for ordinary coverings by rectangles.  The \Defn{refined rectangle
covering number} $\rrc(S_P)$ is the least size of a refined covering. It is
shown in~\cite[Theorem 3.4]{OVW14} that $\rrc(S_P)$ lies between $\rc(S_P)$
and $\rkN(P)$ and thus yields a possibly better lower bound on the nonnegative
rank. We will work with the following relaxation that we call the
\Defn{generic} refined rectangle covering number. We consider coverings by
ordinary rectangles with the additional condition that for any pair of indices
$(i,k),(j,l)$ such that
\begin{equation}\label{eqn:grrc}
    S_{ik}, S_{jl},S_{il}, S_{jk} > 0 \quad \text{ and } \quad
    S_{ik} S_{jl} \ \neq \ S_{il}S_{jk},
\end{equation}
the four entries $S_{ik}, S_{jl},S_{il}, S_{jk}$ are covered by at least two
rectangles. The denomination `generic' is explained in the proof of
Theorem~\ref{thm:52} below.

As for the rectangle covering number, to determine if there is a generic
refined covering of a given size can be phrased as a Boolean formula.  For the
example given above and the number of rectangles set to $9$, a \texttt{python}
script in the appendix produces such a formula with $450$ Boolean variables
and $16796$ clauses.  Any suitable SAT solver verifies that this formula is
unsatisfiable which proves that the $(5,2)$-hypersimplex given
in~\eqref{eqn:52special} has nonnegative rank $10$.

\begin{thm}\label{thm:52}
    The combinatorial $(5,2)$-hypersimplices with nonnegative rank $10$ form a
    dense open subset of $\Rel_{5,2}$.
\end{thm}
\begin{proof}
    Let $\mathcal{I}$ be the collection of all pairs of indices $(i,k;j,l)$
    satisfying~\eqref{eqn:grrc} in the realization of $\Delta_{5,2}$ given
    in~\eqref{eqn:52special}. The set of realizations $X \subseteq \Rel_{5,2}$
    for which~\eqref{eqn:grrc} is not satisfied for some $(i,k;j,l) \in
    \mathcal{I}$ form an algebraic subset of $\Rel_{5,2}$. Since the
    realization~\eqref{eqn:52special} is not in $X$, this shows that 
    $X$ is a proper subset and hence of measure zero in the open set
    $\Rel_{5,2}$.
\end{proof}

The same result extends easily to the remaining $(n,2)$-hypersimplices.
\begin{repcor}{cor:dense2}
    For $n \ge 5$, the combinatorial $(n,2)$-hypersimplices with extension
    complexity $2n$ are dense in $\Rel_{n,2}$.
\end{repcor}
\begin{proof}
 For $n=5$, this is the previous result. For $n\geq 6$, this follows from Theorem~\ref{thm:main_generic}
with the fact that $FG$-generic hypersimplices
 form a dense open subset of $\Rel_{n,2}$, because $G$-genericity is equivalent to the determinant of~\eqref{eq:ProjDeltan2} being non-zero (Lemma~\ref{lem:G_generic}). 
\end{proof}

As of now, we are unable to extend this result to higher values of $k$.
Nevertheless, we conjecture that except for a subset of zero measure, all
$(n,k)$-hypersimplices have nonnegative rank~$2n$ (for $n\geq 5$).  (Notice
that the existence of particular instances already ensures the existence of
open neighborhoods of hypersimplices of nonnegative rank~$2n$.)

\begin{repconj}{conj:dense}
    For $n\geq 5$ and $2\leq k\leq n-2$, the combinatorial hypersimplices of
    nonnegative rank~$2n$ form a dense open subset of $\Rel_{n,k}$.
\end{repconj}

Theorem~\ref{thm:main_generic} implies that combinatorial hypersimplices of
nonnegative rank~$2n$ are an open subset of $\Rel_{n,k}$. However, notice that
$\Rel_{n,k}$ is a quotient of the variety of vanishing principal $k$-minors, which 
has irreducible components of different dimensions~\cite{Wheeler15}. 
We cannot certify that the nonnegative rank~$2n$ subset is dense because it could skip 
a whole component.

\textbf{Acknowledgements.} We would like to thank Stefan Weltge, for help with
the computation of rectangle covering numbers and Günter Ziegler for
insightful discussions regarding realizations of hypersimplices. We are
indebted to Francisco Santos for extensive discussions regarding realization
spaces of hypersimplices. The coordinates for the example in
Proposition~\ref{prop:singular62} are due to him. Finally, we would like to
thank the anonymous reviewers for their careful reading and useful
suggestions; in particular, for pointing out the argument
from~\cite[Lemma~3.3]{FKPT13} which simplified and strengthened
Proposition~\ref{prop:rcbounds}.

\bibliographystyle{amsalpha}
\bibliography{HypersimplexXC-RS}

\end{document}